\documentclass[11pt,reqno]{amsart}

\usepackage{xcolor}
\usepackage{amsmath, amsthm, amssymb, amsxtra}
\usepackage{amsfonts}
\usepackage[utf8]{inputenc}
\usepackage[dvips]{epsfig}
\usepackage{graphicx}
\usepackage[english]{babel}
\usepackage{hyperref}
\usepackage{tablefootnote}

\usepackage{graphicx}
\usepackage{latexsym}
\usepackage{mathtools}
\usepackage{esint}
\usepackage{float}

\theoremstyle{plain}
\newtheorem{thm}{Theorem}[section]
\newtheorem{cor}[thm]{Corollary}
\newtheorem{lem}[thm]{Lemma}
\newtheorem{prop}[thm]{Proposition}

\newtheorem{conj}[thm]{Conjecture}

\theoremstyle{definition}

\theoremstyle{remark}
\newtheorem{obs}[thm]{Remark}

\numberwithin{equation}{section}

\newcommand{\average}{{\mathchoice {\kern1ex\vcenter{\hrule height.4pt
width 6pt depth0pt} \kern-9.7pt} {\kern1ex\vcenter{\hrule
height.4pt width 4.3pt depth0pt} \kern-7pt} {} {} }}

\newcommand{\R}{\mathbb R}
\newcommand{\N}{\mathbb N}

\newcommand{\p}{\partial}

\newcommand{\comment}[1]{}
\newcommand{\pfirst}{^{(1)}}
\newcommand{\psecond}{^{(2)}}
\newcommand{\pthird}{^{(3)}}
\newcommand{\pj}{^{(j)}}
\newcommand{\loc}{_{\mathrm{loc}}}
\newcommand{\dimH}{\operatorname{dim}_{\mathcal{H}}}
\newcommand{\Pdos}{\mathcal{P}_2}

\begin{document}

\title[Generic regularity for the thin obstacle problem]{Generic regularity of free boundaries for\\the thin obstacle problem}
\author{Xavier Fernández-Real}
\address{EPFL SB, Station 8, CH-1015 Laussane, Switzerland.}
\email{\tt xavier.fernandez-real@epfl.ch}
\author{Clara Torres-Latorre}
\address{Universitat de Barcelona, Departament de Matem\`atiques i Inform\`atica, Gran Via de les Corts Catalanes 585, 08007 Barcelona, Spain.}
\email{\tt claratorreslatorre@ub.edu}


\begin{abstract}
The free boundary for the Signorini problem in $\R^{n+1}$ is smooth outside of a degenerate set, which can have the same dimension ($n-1$) as the free boundary itself.

In \cite{FR21} it was shown that \textit{generically}, the set where the free boundary is not smooth is at most $(n-2)$-dimensional. Our main result establishes that, in fact, the degenerate set has zero $\mathcal{H}^{n-3-\alpha_0}$ measure for a generic solution. As a by-product, we obtain that, for $n+1 \leq 4$, the whole free boundary is generically smooth. This solves the analogue of a conjecture of Schaeffer in $\R^3$ and $\R^4$ for the thin obstacle problem.
\end{abstract}

\thanks{X.F. was supported by the SNF grants 200021\_182565 and PZ00P2\_208930,  by the Swiss State Secretariat for Education, Research and Innovation (SERI) under contract number MB22.00034, and by the AEI project PID2021-125021NA-I00 (Spain). C.T. has received funding from the European Research Council (ERC) under the Grant Agreement No 801867, from the grant RED2018-102650-T funded by MCIN/AEI/10.13039/501100011033, and from AEI project PID2021-125021NAI00 (Spain).}
\maketitle


\section{Introduction}\label{sect:intro}

The Signorini problem (also known as the thin or boundary obstacle problem) is a classical free boundary problem that was originally studied by Antonio Signorini in connection with linear elasticity \cite{Sig33, Sig59, KO88}. The same equations appear in a variety of settings such as Biology, Fluid Mechanics, and Finance, and they have received a lot of interest from different areas \cite{DL76, Mer76, CT04, Ros18, Fer22}.

The thin obstacle problem is equivalent to the obstacle problem for the half-Laplacian $(-\Delta)^{1/2}$, and has been extensively studied by the mathematical community in the last two decades; see \cite{Caf79,AC04,CS07,ACS08,GP09,PSU12,KPS15,DS16,DGPT17,FS18,KRS19,CSV20,Shi20,FJ21, FS21}, and the references therein. In particular, the study of the Signorini problem is a crucial ingredient to understand the free boundary in the \textit{thick} obstacle problem \cite{FS19,FRS20, SY21, SY22}.

Obstacle problems belong to a wide class of problems known as \textit{free boundary problems}, where one of the unknowns is the contact set, and more precisely, its boundary, the free boundary. There are explicit constructions \cite{Sch76} for the classical obstacle problem that give rise to free boundaries having a set of singular points of the same dimension as the whole free boundary. Still, singular points are expected to be \textit{infrequent}: Schaeffer conjectured in 1974 (\cite{Sch74}) that, for a generic boundary datum, the free boundary is regular. The conjecture was proved to hold true in the plane $\R^2$ by Monneau in \cite{Mon03}, and much more recently in a breakthrough work, \cite{FRS20}, Figalli, Ros-Oton, and Serra showed that it also holds in $\R^3$ and $\R^4$.

Given the parallels between the classical obstacle problem and the thin obstacle problem, it is natural to extend the conjecture of Schaeffer to the setting of the latter:

\begin{conj}\label{conj:schaeffer}
Generically, the free boundary in the Signorini problem is smooth.
\end{conj}

Also for the thin obstacle problem, there are examples of particular solutions having non-regular points of the same dimension as the whole free boundary (see e.g. \cite{GP09, FR21}). The validity of the previous conjecture would imply that such solutions are \emph{rare}. 

Conjecture \ref{conj:schaeffer} was recently proved in $\R^2$ by the first author and Ros-Oton in \cite{FR21} (with operators ${\rm div}(|x_{n+1}|^a\nabla\cdot)$ for $a\in (-1,1)$). In this work, we will extend its validity to the physical dimension $\R^3$, and $\R^4$. Moreover, we will also provide dimensional estimates for the size of the set of degenerate points for dimensions $n+1 \geq 5$.

\subsection{The Signorini problem and the free boundary}
The Signorini problem with zero obstacle (originally introduced as the Laplace equation with ambiguous boundary conditions) can be written as 
\begin{equation}\label{eq:signorini_onesided}
\left\{\begin{array}{rclll}
\Delta u & = & 0 & \text{in} & B_1^+\\
\min\{u,-\partial_{x_{n+1}} u\} & = & 0 & \text{on} & B_1\cap\{x_{n+1} = 0\}.
\end{array}\right.
\end{equation}

Alternatively, we study the problem posed in the whole ball $B_1\subset \R^{n+1}$ (extending by even symmetry) as 
\begin{equation}\label{eq:signorini}
\left\{\begin{array}{rclll}
\Delta u & = & 0 & \text{in} & B_1 \setminus \{x_{n+1} = 0\}\\
\min\{u,-\Delta u\} & = & 0 & \text{on} & B_1\cap\{x_{n+1} = 0\}\\
u(x',x_{n+1}) & = & u(x',-x_{n+1}) & \text{in} & B_1,
\end{array}\right.
\end{equation}
where now $\Delta u$ needs to be understood in the sense of distributions. For the Signorini problem, solutions are always $C^{1,1/2}$ (on each side in \eqref{eq:signorini}, see \cite{AC04}).


Like the obstacle problem, the Signorini problem is a free boundary problem. That is, one of the unknowns of the problem is the contact set
$$\Lambda(u) := \{x' \in \R^n : u(x',0) = 0\}\times\{0\},$$
and in particular, its boundary (in the relative topology on the thin space), the \textit{free boundary}
$$\Gamma(u) := \p\{x'\in\R^n : u(x',0)=0\}\times\{0\}.$$

The free boundary has been mainly studied so far by means of blow-up methods. Namely, assume that $u$ is a solution to (\ref{eq:signorini}) with $0 \in \Gamma(u)$, and define the blow-up sequence
\begin{equation}
    \label{eq:rescalings}
u_r(x) := \frac{u(rx)}{\|u\|_{L^2(\p B_r)}}.
\end{equation}
It can be shown that, up to a subsequence $r_k\downarrow 0$, $u_r$ converges (locally uniformly) to a global $\kappa$-homogeneous solution $u_0$. The value $\kappa$ is what we call the order or frequency of the free boundary point.

The free boundary is divided into \textit{regular points}, $\operatorname{Reg}(u)$ (with homogeneity $\kappa = 3/2$), and \textit{degenerate points}, $\operatorname{Deg}(u)$ (with homogeneity $\kappa \geq 2$), \cite{ACS08}:
\[
\Gamma(u) = \operatorname{Reg}(u) \cup \operatorname{Deg}(u).
\]
Moreover, \textit{for almost every solution}, the dimension of the set of degenerate points is at most $n - 2$, so they are \textit{rare} \cite{FR21}. We refer to \cite{PSU12, Fer22} for more details about the structure of the free boundary, and the thin obstacle problem in general.

\subsection{Main results}
We prove that \textit{generically}, the set of degenerate points is empty in dimensions $n+1 = 3$ and $n+1 = 4$. More precisely, we consider monotone families of solutions as follows. 

Let $u : B_1\times[-1,1]\to \R$ be such that $u(\cdot,t)$  solves \eqref{eq:signorini} for each $t\in [-1, 1]$ and
\begin{equation}\label{eq:monotonicity}
\left\{
\begin{array}{rclll}
u(\cdot,t') - u(\cdot,t) & \geq & 0 & \text{in} & \overline{B_1}\\
u(\cdot,t') - u(\cdot,t) & \geq & t' - t & \text{on} & \p B_1 \cap \{|x_{n+1}| \geq \frac{1}{2}\}\\
\|u(\cdot,t)\|_{C^{0,1}(B_1)} & \leq & 1,
\end{array}
\right.
\end{equation}
for all $-1 \leq t < t' \leq 1$. As there is no room for confusion, we will say that $u : B_1\times[-1,1]\to\R$ solves (\ref{eq:signorini}) if $u(\cdot,t)$ solves it for all $t \in [-1,1]$. Our main result is the following:

\begin{thm}\label{thm:main}
Let $u : B_1\times[-1,1]\to\R$ be a solution to \eqref{eq:signorini}-\eqref{eq:monotonicity}. Then, for almost every $t \in [-1,1]$,
\begin{itemize}
    \item[(a)] If $n \leq 3$, $\operatorname{Deg}(u(\cdot,t)) = \emptyset$.
    \item[(b)] If $n \geq 4$, $\dimH(\operatorname{Deg}(u(\cdot,t))) \le n - 3-\alpha_\circ$, for some $\alpha_\circ > 0$ depending only on $n$.
\end{itemize}
\end{thm}

Here, $\dimH$ denotes the Hausdorff dimension of a set; see for example \cite[Chapter 4]{Mat95}. We actually prove stronger results for several subsets of the free boundary, see Proposition \ref{prop:generic}. See also subsection~\ref{ss:sketch} for a sketch of the proof of Theorem~\ref{thm:main}.

As a consequence of our main result we obtain that, generically, free boundaries are smooth in $\R^3$ and $\R^4$, thus showing that the analogue of Schaeffer's conjecture for the thin obstacle problem holds true in these dimensions.

\begin{cor}\label{cor:schaeffer}
Conjecture \ref{conj:schaeffer} holds in $\R^3$ and $\R^4$.
\end{cor}

We recall that this was only known in $\R^2$, \cite{FR21}.

\begin{obs}\label{obs:prevalence}
The notion of genericity needs to be understood in the context of the theory of prevalence, \cite{HSY92} (see also \cite{OY05}). In this language, we will prove that the set of solutions satisfying that the free boundary has an empty degenerate set is \textit{prevalent} within the set of solutions in $\R^3$ and $\R^4$ (say, given by $C^0$ or $L^\infty$ boundary data). Alternatively, we   show that the set of solutions whose degenerate set is non-empty is \textit{shy}. In particular, this means that for almost every boundary data (see \cite[Definition 3.1]{OY05}) the corresponding solution has a smooth free boundary (by \cite{KPS15, DS16}).
\end{obs}

\begin{obs}
 The result in Corollary~\ref{cor:schaeffer} is in correspondence with the results in the \emph{thick} case in \cite{FRS20}, in $\R^3$ and $\R^4$ as well. Part of the appeal of the present manuscript is that, due to the nature of the problem, the methods developed in \cite{FRS20} become much simpler in the context of the Signorini problem  (once combined with \cite{CSV20, FR21, FJ21, SY21}), allowing us to obtain an equally strong result with far fewer technical details. Indeed, in our case, the free boundary is a set of co-dimension 2 (instead of co-dimension 1), making it a set of zero harmonic capacity. This implies, in particular, that the second-order expansion around singular points is harmonic (see Propositions~\ref{prop:ff_increasing_w} and \ref{prop:quadratic_stratification}). Conversely, in the \emph{thick case}, the second-order term in the expansion around singular points can have different behaviors, one of them being, precisely, a solution to a thin obstacle problem, that also needs to account for the curvature of the contact set around the point, and has a different thin space at each point. Roughly speaking, the role played by $u-p$ in the thick case (where $p$ is the first order expansion around a free boundary point, that depends on the point), is now played by $u$ (which is the same at all points, thus allowing for a simpler analysis).


In the same way, this also means that the dimension in which Conjecture \ref{conj:schaeffer} holds cannot be improved only using the approach in \cite{FRS20}. (More specifically, completely new ideas are needed to improve the generic size of the set $\Gamma_2^{\mathrm{a}}(u)$; see subsections~\ref{ss:quadratic} and~\ref{ss:sketch}.) 
\end{obs}



\begin{obs}
In this work, we deal with the Signorini problem with zero obstacle, \eqref{eq:signorini_onesided} or \eqref{eq:signorini} (as in \cite{CSV20, FJ21, SY21}), which is a model case including the problem with an analytic obstacle. 

Indeed, given a function $\varphi:B_1'\subset \R^n \to \R$  where $B_1'$ denotes the unit ball in $\R^n$, the Signorini problem with obstacle $\varphi$ is 
\[
\left\{\begin{array}{rclll}
\Delta u & = & 0 & \text{in} & B_1\cap\{x_{n+1} > 0\}\\
\min\{u(x', 0)-\varphi(x'),-\partial_{x_{n+1}} u(x', 0)\} & = & 0 & \text{for} & x'\in B'_1.
\end{array}\right.
\]
When $\varphi$ is analytic, it can be extended to a harmonic function  in $B_1\subset \R^{n+1}$  (i.e., with $\tilde\varphi(x', 0) = \varphi(x')$ for all $x'\in B_1$), even in the last coordinate, so that $v := u -\tilde\varphi$ is a solution to the Signorini problem with zero obstacle, \eqref{eq:signorini}. That is, our result also applies to analytic obstacles.
\end{obs}

\begin{obs}
Apart from the aforementioned works, \cite{FRS20, FR21}, the recent preprints \cite{FY23} and \cite{CMS23, CMS23b} obtain similar results using related techniques in the context of the Alt-Caffarelli and Alt-Phillips functionals, and minimal surfaces, respectively.
\end{obs}

\subsection{Plan of the paper}
This paper is organized as follows.

In Section \ref{sect:previ} we introduce some technical tools, such as the frequency formula, and some preliminary results. We also sketch the strategy of the proof of Theorem~\ref{thm:main} at the end of the section. Then, the goal of Section \ref{sect:dim_reduction1} is to recover the known dimensional bounds for $\operatorname{Deg}(u)$ and one of its subsets, that we denote $\Gamma_*(u)$ (see \eqref{eq:boldgamma}),  but for a monotone family of solutions (instead of a single solution). In Section \ref{sect:quadratic} we study the points of order $2$, separating them into \textit{ordinary quadratic points}, for which we show an improved cleaning; and \textit{anomalous quadratic points}, for which we perform a further dimension reduction; and in Section \ref{sect:cubic} we study the cubic points. Finally, in Section \ref{sect:conclusion} we combine our results to compute the final dimensional estimates.

\section{Preliminaries}\label{sect:previ}
In this section we recall some background results and we develop some technical tools that will be useful later. We start with the following Liouville-type result.
\begin{lem}\label{lem:global_signorini}
Let $u : \R^{n+1} \to \R$ be a $\kappa$-homogeneous solution to (\ref{eq:signorini}). Then,
\begin{itemize}
    \item[(a)] If $u \geq 0$, then $u \equiv 0$.
    \item[(b)] If $u \leq 0$ and $\kappa > 1$, then $u \equiv 0$.
    \item[(c)] If $\p_e u \geq 0$ for some direction $e$ and $\kappa \geq 2$, then $u$ is invariant in the direction $e$.
\end{itemize}
\end{lem}

\begin{proof}
(a) Suppose $u$ is not identically zero. Then, by the Hopf lemma $\p_{n+1}u(0,0^+) > 0$, which together with $u$ being even in the $x_{n+1}$ direction contradicts the fact that $u$ is superharmonic across the thin space $\{x_{n+1}=0\}$.

(b) Suppose $u$ is not identically zero. Then, by the Hopf lemma $\p_{n+1}u(0,0^+) < 0$. On the other hand, $\nabla u(0) = 0$ because the homogeneity of $u$ is $\kappa > 1$. A contradiction.

(c) First, $\p_eu(0) = 0$ because $\kappa \geq 2$. Assume by contradiction that $\p_e u > 0$ in $\{x_{n+1} > 0\}$, and thus by the Hopf lemma $\p_{n+1}\p_e u(0,0^+) > 0$. Therefore, $D^2u(0)\not\equiv0$, which in turn implies $\kappa = 2$, and it follows by \cite[Theorem 3]{ACS08} that $u(x) = \sum_{i=1}^na_i(x_i^2-x_{n+1}^2)$ with $a_i \geq 0$, after a change of coordinates if necessary. Hence, $\p_e u$ is linear and since $\p_e u\ge 0$, we get $\p_e u\equiv 0$, a contradiction.
%
\end{proof}

We continue with a Hopf-type estimate to quantify the monotonicity of the families of solutions near the thin space.
\begin{lem}\label{lem:hopf}
Let $u : B_1\times[-1,1]\to\R$ be a solution to (\ref{eq:signorini})-(\ref{eq:monotonicity}). Then, for all $t \geq 0$,
$$h_t(x) := u(x,t) - u(x,0) \geq  ct|x_{n+1}| \ \text{in} \ B_{1/2},$$
for some $c > 0$ depending only on $n$.
\end{lem}






\begin{proof}
By (\ref{eq:monotonicity}), $h_t \geq 0$ in $B_1$, and $h_t \geq t$ on $\p B_1\cap\{|x_{n+1}| \geq \frac{1}{2}\}$. Let $\varphi$ be such that $\varphi = 1$ on $\partial B_1\cap \{|x_{n+1}|\ge \frac12\}$, $\varphi = 0$ on $\partial B_1\cap \{|x_{n+1}|< \frac12\}$ and $\{x_{n+1} = 0\}$, and $\Delta \varphi = 0$ in $B_1\cap \{x_{n+1}\neq 0\}$. Then, on the one hand, thanks to the Hopf Lemma we have that $\varphi \ge c|x_{n+1}|$ in $B_{1/2}$ for some $c$ depending only on $n$; and on the other hand, by the maximum principle, $\varphi \le \frac{h_t}{t}$ in $B_1$.
\end{proof}

Given $u : B_1\times[-1,1]\to\R$ a family of solutions of (\ref{eq:signorini})-(\ref{eq:monotonicity}), we define the free boundary
$$\Gamma(u(\cdot,t)) = \p\{x' \in \R^n : u((x',0),t) = 0\}\times\{0\},$$
and we denote
$$\mathbf{\Gamma} := \bigcup\limits_{t\in[-1,1]}\Gamma(u(\cdot,t)).$$

Analogously, we will denote by $\mathbf{\operatorname{Reg}}$ and $\mathbf{\operatorname{Deg}}$ the union of all regular and degenerate points for a family of solutions. For our setting, it is convenient to define the following map:

\begin{prop}[\protect{\cite[Corollary 2.7]{FR21}}]\label{prop:tau_cont}
Let $u : B_1\times[-1,1]\to\R$ be a solution to (\ref{eq:signorini})-(\ref{eq:monotonicity}). Then, the mapping $\tau : \mathbf{\Gamma} \to [-1,1]$ defined as $\tau(x_0) = t_0$ such that $x_0 \in {\Gamma}(u(\cdot,t_0))$ is well defined and continuous. Moreover, for any $\varepsilon > 0$, the map
$$\mathbf{\Gamma}\cap B_{1-\varepsilon} \ni x_0 \mapsto u(x_0+\cdot,\tau(x_0))$$
is continuous in the $C^0$ norm.
\end{prop}

\subsection{The frequency formula}
Here, we recall and prove some facts about Almgren's frequency function.

Given $w \in H^1\loc$, we define
$$\phi(r,w) := \frac{D(r,w)}{H(r,w)},$$
where
$$D(r,w) := r^{1-n}\int_{B_r}|\nabla w|^2 \quad \text{and} \quad H(r,w) := r^{-n}\int_{\p B_r}w^2.$$

We recall that the frequency function $\phi$ is nondecreasing in $r$ for solutions of (\ref{eq:signorini}):

\begin{lem}[\protect{\cite[Lemma 1]{ACS08}}]\label{lem:ff_increasing}
Let $u$ be a solution to (\ref{eq:signorini}). Then, the function $r \mapsto \phi(r,u)$ is nondecreasing. Moreover, $\phi(r,u)$ is constant with respect to $r$, $\phi(r,u)\equiv \lambda$, if and only if $u$ is $\lambda$-homogeneous.
\end{lem}
This justifies that the frequency of a point $x_0$, $\phi(0^+,u(x_0+\cdot))$, is always well defined; and hence, we can stratify the free boundary according to the frequency $\kappa$ as follows (see Proposition \ref{prop:tau_cont}):
$$\Gamma_\kappa(u(\cdot,t)) := \{x_0 \in \Gamma(u(\cdot,t)) : \phi(0^+,u(x_0+\cdot,t)) = \kappa\},\quad \mathbf{\Gamma}_\kappa := \bigcup\limits_{t\in[-1,1]}\Gamma_\kappa(u(\cdot,t)),$$
and we also introduce the sets
\begin{equation}
    \label{eq:boldgamma}
\begin{split}
\Gamma_{\geq\kappa}(u(\cdot,t)) := \bigcup\limits_{\nu\geq\kappa}\Gamma_{\nu}(u(\cdot,t)), &\quad\mathbf{\Gamma}_{\geq\kappa} := \bigcup\limits_{\nu\geq\kappa}\mathbf{\Gamma}_{\nu},\\
\Gamma_*(u(\cdot,t)) := \bigcup\limits_{\nu \in \R\setminus S}\Gamma_\nu(u(\cdot,t)), &\quad\mathbf{\Gamma}_* := \bigcup\limits_{\nu\in\R\setminus S}\mathbf{\Gamma}_\nu,
\end{split}
\end{equation}
where $S = \{1,\frac{3}{2},2,3,\frac{7}{2},4,\ldots\} = \N\cup \{2\N-\frac12\}$ is the set of possible homogeneities of the solutions of Signorini in dimension $n+1 = 2$.

Observe that the frequency function can act as a proxy for the growth rate of a function:
\begin{lem}\label{lem:H_bounds}
Let $u : B_1 \to \R$ be a solution to (\ref{eq:signorini}). Suppose that for $0 < r < R < 1$ we have $\underline{\lambda} \leq \phi(r,u) \leq \phi(R, u) \leq \overline{\lambda}$. Then,
$$\left(\frac{R}{r}\right)^{2\underline{\lambda}} \leq \frac{H(R,u)}{H(r,u)} \leq \left(\frac{R}{r}\right)^{2\overline{\lambda}}.$$
\end{lem}

\begin{proof}
Let $u_r := u(r\cdot)$. Then, $H(r,u) = \int_{\p B_1}u_r^2$, and integrating by parts,
$$H'(r,u) = \frac{2}{r}\int_{\p B_1} u_r\p_\nu u_r = \frac{2}{r}\left(\int_{B_1}|\nabla u_r|^2 + \int_{B_1}u_r\Delta u_r\right) = \frac{2}{r}D(u,r),$$
because $u_r\Delta u_r = 0$ for solutions of (\ref{eq:signorini}), and hence
$$\frac{H'(r,u)}{H(r,u)} =  \frac{2}{r}\phi(r,u).$$
Then, integrating from $r$ to $R$ (and since $\phi$ is monotone nondecreasing, see Lemma~\ref{lem:ff_increasing}),
$$2\underline{\lambda}\ln(R/r) \leq \ln\left(\frac{H(R,u)}{H(r,u)}\right) \leq 2\overline{\lambda}\ln(R/r),$$
and the conclusion follows.
\end{proof}

Finally, once the frequency is properly defined, we may recall two results that will be used later. The first one is a strong comparison principle, from which we copy the proof for the convenience of the reader. 
\begin{lem}[\protect{\cite[Lemma A.4]{FRS20}}]\label{lem:FRSA.4}
Let $u, v$ be two solutions of (\ref{eq:signorini}) satisfying $u \geq v$ in $B_1$ and $u(0) = v(0) = 0$. If $\phi(0^+,v) > 1$ or $v \equiv 0$, then $u \equiv v$.
\end{lem}

\begin{proof}
Assume by contradiction that $u \not\equiv v$. Then, $u > v$ in $\{x_{n+1} > 0\}$, and by the Hopf lemma $\p_{n+1}(u-v)(0,0^+) > 0$. On the other hand, since $\phi(0^+,v) > 1$ or $v \equiv 0$, $\nabla v(0) = 0$, and it follows that $\p_{n+1}u(0,0^+) > 0$, and since $\Delta u = 2\p_{n+1}u\mathcal{H}^n|_{\{x_{n+1}=0\}}$ distributionally, this contradicts the fact that $\Delta u \leq 0$.
\end{proof}

The second one is the following cleaning result.
\begin{prop}[\protect{\cite[Propositions 2.4 \& 2.9]{FR21}}]\label{prop:cleaning_old}
Let $u : B_1\times[-1,1]\to\R$ be a solution to (\ref{eq:signorini})-(\ref{eq:monotonicity}). Let $\delta > 0$ small, and let $x_0 \in B_{1-\delta}\cap\Gamma_{\geq \kappa}(u(\cdot,t_0))$. Then, there exists $\rho > 0$ such that
$$\{(x,t) \in B_\rho(x_0)\times[-1,1] : t > t_0+C|x-x_0|^{\kappa - 1}\}\cap\{u = 0\}\cap\{x_{n+1}=0\} = \emptyset,$$
for some constant $C$ depending only on $n$, $\kappa$ and $\delta$. Moreover, if $\kappa = 2$, for every $\varepsilon > 0$ there exists $\rho > 0$ such that
$$\{(x,t) \in B_\rho(x_0)\times[-1,1] : t > t_0+C|x-x_0|^{2-\varepsilon}\}\cap\{(x,t) : x \in \Gamma_2(u(\cdot,t))\} = \emptyset,$$
for some constant $C$ depending only on $n$ and $\varepsilon$.
\end{prop}

%
%
%


\subsection{Quadratic points}
\label{ss:quadratic}
Given $u$ a solution to (\ref{eq:signorini}) and a singular point $x_0 \in \Gamma_2(u)$, we denote by $p_{2,x_0}$ the first blow-up\footnote{Observe that these are not rescalings that preserve the $L^2(\partial B_1)$ norm (cf. the sequence \eqref{eq:rescalings}). In fact, at singular points both types of rescalings coincide up to a multiplicative constant. By rescaling directly by $r^2$ we obtain the first order expansion of $u$, that is, $u(x_\circ + \cdot) = p_{2,x_\circ}(x) + o(|x|^2)$.} of $u$ at $x_0$,
\begin{equation}
    \label{eq:p2def}
    p_{2,x_0} := \lim\limits_{r \rightarrow 0}\frac{u(x_0+r\cdot)}{r^2}.
\end{equation}
This expression is uniquely defined by \cite[Theorem 1.3.6 or Theorem 1.5.4]{GP09}, and $p_{2,x_0} \equiv 0$ if and only if $x_0\in \Gamma_{>2}(u)$ (by \cite[Lemmas 1.5.1 and 1.5.2]{GP09}). The blow-up $p_{2,x_0}$ belongs to the set of homogeneous quadratic harmonic even polynomials that are nonnegative on the thin space, i.e.
$$\Pdos := \{p : \Delta p = 0, \ x\cdot\nabla p = 2p, \ p(x',0) \geq 0, \ p(x',x_{n+1}) = p(x',-x_{n+1})\}.$$
Notice how $p = 0$ also belongs to $\Pdos$.

The following proposition will allow us to perform a second blow-up at the points of frequency $2$ to attain a finer understanding of singular points:
\begin{prop}[\protect{\cite[Proposition 2.2]{FJ21}}]\label{prop:ff_increasing_w}
Let $u$ be a solution to (\ref{eq:signorini}), and assume that $0\in \Gamma_{\ge 2}(u)$ (i.e. $\phi(0^+,u) \ge 2$). Let $p \in \Pdos$ and let $w := u - p$. Then, the function $r \mapsto \phi(r,w)$ is nondecreasing, and its derivative satisfies
$$\phi'(r,w) \geq \frac{2}{r}\left(\frac{\int_{B_1}w_r\Delta w_r}{\int_{\p B_1}w_r^2}\right)^2,$$
with $w_r(x) := w(rx)$. Moreover, $\phi(0^+,w) \geq 2$.
\end{prop}

\begin{proof}
This result corresponds to \cite[Proposition 2.2]{FJ21} in combination with the computations inside its proof.
\end{proof}

The following lemma asserts that the $L^2$ rate of growth of $u - p$ can be estimated by its frequency (cf. Lemma \ref{lem:H_bounds}).

\begin{lem}\label{lem:H_bounds_2}
Let $u : B_1 \to \R$ be a solution to (\ref{eq:signorini}), and assume that $0\in \Gamma_2(u)$ (i.e. $\phi(0^+,u) = 2$). Let $p \in \Pdos$. Suppose that for $0 < r < R < 1$ we have $\underline{\lambda} \leq \phi(r,u - p) \le \phi(R, u-p)\leq \overline{\lambda}$. Then, for any given $\delta > 0$,
$$\left(\frac{R}{r}\right)^{2\underline{\lambda}} \leq \frac{H(R,u-p)}{H(r,u-p)} \leq C_\delta\left(\frac{R}{r}\right)^{2\overline{\lambda}+\delta},$$
where $C_\delta$ depends only on $\delta$, $\overline{\lambda}$, and the dimension.
\end{lem}

\begin{proof}
First, we define $w := u - p$, $w_r := w(r\cdot)$, and
$$F(r,w) := \frac{r^{1-n}\int_{B_r}w\Delta w}{r^{-n}\int_{\p B_r}w^2} = \frac{\int_{B_1}w_r\Delta w_r}{\int_{\p B_1}w_r^2}.$$
Since $p \geq 0$ on the thin space, and $\Delta u = 0$ outside of it, $w\Delta w = -p\Delta u \geq 0$.

Observe that
$$H'(r,w) = \frac{2}{r}\int_{B_1}|\nabla w_r|^2 + \frac{2}{r}\int_{B_1}w_r\Delta w_r \quad \Rightarrow \quad \frac{H'(r,w)}{H(r,w)} = \frac{2}{r}\left(\phi(r,w) + F(r)\right).$$
Integrating, we get
$$\ln\left(\frac{H(R,w)}{H(r,w)}\right) = \int_r^R\frac{2}{\rho}\left(\phi(\rho,w)+F(\rho,w)\right)\mathrm{d}\rho.$$

On the one hand, since $F(\rho,w) \geq 0$ and $\phi$ is nondecreasing (by Proposition~\ref{prop:ff_increasing_w}),
$$\ln\left(\frac{H(R,w)}{H(r,w)}\right) \geq \int_r^R\frac{2}{\rho}\phi(\rho,w)\mathrm{d}\rho \geq 2\underline{\lambda}\ln(R/r),$$
and the inequality in the left follows. On the other hand, using Proposition \ref{prop:ff_increasing_w},
$$\int_r^RF(\rho,w)\frac{\mathrm{d}\rho}{\rho} \leq \left(\int_r^RF(\rho,w)^2\frac{\mathrm{d}\rho}{\rho}\right)^{1/2}\left(\int_r^R\frac{\mathrm{d}\rho}{\rho}\right)^{1/2} \leq \left(\frac{\overline{\lambda}-\underline{\lambda}}{2}\right)^{1/2}\ln(R/r)^{1/2},$$
and then
$$\ln\left(\frac{H(R,w)}{H(r,w)}\right) = \int_r^R\frac{2}{\rho}\left(\phi(\rho,w)+F(\rho,w)\right)\mathrm{d}\rho \leq 2\overline{\lambda}\ln(R/r) + C\ln(R/r)^{1/2},$$
so that the conclusion follows by the estimate $\sqrt{t} \leq \delta t + C_\delta$.
\end{proof}

By means of Proposition~\ref{prop:ff_increasing_w}, quadratic free boundary points can be further stratified in terms of a second blow-up. That is, if $x_0\in \Gamma_2(u)$, we define the second blow-up sequence
$$\tilde{w}_r := \frac{u(x_0+r\cdot)-p_{2,x_0}(r\cdot)}{\|u(x_0+r\cdot)-p_{2,x_0}(r\cdot)\|_{L^2(\p B_1)}},$$
which converges to a $\lambda$-homogeneous function with $\lambda = \phi(0^+,u(x_0+\cdot)-p_{2,x_0})$, up to a subsequence, thanks to the monotonicity of $\phi$ along $u-p$ given by Proposition~\ref{prop:ff_increasing_w}:
\begin{prop}[\protect{\cite[Proposition 3.2]{FJ21}}]\label{prop:quadratic_stratification}
For every sequence $r_j \downarrow 0$, there is a subsequence $r_{j_l} \downarrow 0$ such that $\tilde{w}_{r_{j_l}} \rightharpoonup q$ in $H^1\loc$, and $q \not\equiv 0$ is a $\lambda$-homogeneous harmonic polynomial, with $\lambda = \phi(0^+, u(x_0+\cdot)-p_{2,x_0}) \in \{2,3,4,\ldots\}$.
\end{prop}

Then, we define the ordinary and anomalous quadratic points as follows:
\begin{equation}
\label{eq:def_ord_anom}
\begin{split}
    \Gamma_2^{\mathrm{o}}(u) &:= \{x_0 \in \Gamma_2(u) : \phi(0^+, u(x_0+\cdot)-p_{2,x_0}) \geq 3\}\\
    \Gamma_2^{\mathrm{a}}(u) &:= \{x_0 \in \Gamma_2(u) : \phi(0^+, u(x_0+\cdot)-p_{2,x_0}) = 2\},   
\end{split}
\end{equation}
and we define the sets $\mathbf{\Gamma}_2^{\mathrm{o}}$ and $\mathbf{\Gamma}_2^{\mathrm{a}}$ analogously for a family of solutions (cf. \eqref{eq:boldgamma}). Ordinary quadratic points are called \textit{generic quadratic points} in \cite{FJ21}, but we have decided to change the terminology in order to avoid confusion.

The second blow-up satisfies the following orthogonality property with the first one, coming from an optimality condition:

\begin{lem}[\protect{\cite[Lemma 3.3]{FJ21}}]\label{lem:orthogonality}
Let $u$ be a solution to (\ref{eq:signorini}) with $0 \in \Gamma_2(u)$. Let $p_2\in \Pdos$ be the blow-up of $u$ at 0, and let $q$ be a second blow-up as introduced in Proposition \ref{prop:quadratic_stratification}. Then,
$$\int_{\p B_1}p_2q = 0$$
and
$$\int_{\p B_1}pq \leq 0 \quad \forall p \in \Pdos.$$
\end{lem}

\subsection{Cubic points}

We will take advantage of the following recently improved convergence to the cubic blow-up:
\begin{thm}[\protect{\cite[Theorem 1.1]{SY21}}]\label{thm:cubic_approx}
Let $u$ be a solution to \eqref{eq:signorini} with $0 \in \Gamma_3(u)$ and $\|u\|_{L^\infty(B_1)} \le 1$. Then, there exists  a 3-homogeneous solution to \eqref{eq:signorini}, $p_3$, such that
$$\|u - p_3\|_{L^\infty(B_r)} \leq Cr^{3+\alpha},$$
for some $C, \alpha >0$ depending only on $n$. 
\end{thm}

We will also use the following characterization of global cubic solutions.
\begin{lem}[\protect{\cite[Lemma 5.2]{FRS20}}]\label{lem:cubic_blowup}
Let $p_3$ be a 3-homogeneous solution to (\ref{eq:signorini}). Then,
$$p_3(x) = |x_{n+1}|(ax_{n+1}^2 - x'\cdot Ax'),$$
where $a \geq 0$, $A$ is symmetric and nonnegative definite, and $a = \operatorname{Tr} A$.
\end{lem}

\subsection{Geometric measure theory tools}

We will use the following Reifenberg-type result using the frequency function $\phi$ as $f$, to perform dimension reduction arguments only at the points of continuity of $\phi$.
\begin{lem}[\protect{\cite[Lemma 7.3]{FRS20}}]\label{lem:FRS7.3}
Let $E \subset \R^n$, and $f : E \to \R$. Assume that, for any $\varepsilon > 0$ and $x \in E$, there exists $\rho > 0$ such that, for all $r \in (0,\rho)$,
$$E\cap\overline{B_r(x)}\cap f^{-1}\big([f(x)-\rho,f(x)+\rho]\big) \subset \{y : \operatorname{dist}(y,\Pi_{x,r}) \leq \varepsilon r\},$$
for some $m$-dimensional plane $\Pi_{x,r}$ passing through $x$ (possibly depending on $r$). Then, $\operatorname{dim}_{\mathcal H}(E) \leq m$.
\end{lem}

We will also use the following abstract proposition in the proof of our main result, Theorem \ref{thm:main}, in order to bound the sizes of each of the subsets of the free boundary.
\begin{prop}[\protect{\cite[Corollary 7.8]{FRS20}}]\label{prop:GMT}
Consider the family $\{E_t\}_{t \in [-1,1]}$ with $E_t \subset \R^n$, and let us denote $E := \bigcup\limits_{t \in [-1,1]} E_t$.

Let $1 \leq \beta \leq n$, and assume that the following holds:
\begin{itemize}
    \item $\operatorname{dim}_\mathcal{H} E \leq \beta$,
    \item for all $\varepsilon > 0$, $t_0 \in [-1,1]$, and $x_0 \in E_{t_0}$, there exists $\rho > 0$ such that
    $$B_r(x_0) \cap E_t = \emptyset,$$
    for all $r \in (0,\rho)$ and $t > t_0 + r^{\gamma - \varepsilon}$.
\end{itemize}

Then, 
\begin{itemize}
    \item[(a)] If $\gamma > \beta$, $\dimH(\{t : E_t \neq \emptyset\}) \leq \beta/\gamma$.
    \item[(b)] If $\gamma \leq \beta$, $\dimH (E_t) \leq \beta - \gamma$, for $\mathcal{H}^1$-a.e. $t \in [-1,1]$.
\end{itemize}
\end{prop}

\subsection{Sketch of the proof}
\label{ss:sketch}

The proof is done by combining the ideas and techniques from \cite{FRS20} with the results in \cite{CSV20,FJ21,FR21,SY21}.

The two key parts of our strategy are dimension reduction arguments for families of solutions and cleaning lemmas combined with Proposition \ref{prop:GMT}. We then apply the two steps to different subsets of the free boundary, using the following stratification:
$$\operatorname{Deg}(u) = \Gamma_2^{\mathrm{o}}(u)\cup\Gamma_2^{\mathrm{a}}(u)\cup\Gamma_3(u)\cup\Gamma_{\geq 7/2}(u)\cup\Gamma_*(u).$$

First, given a family of solutions $u : B_1\times[-1,1]\to\R$ to (\ref{eq:signorini})-(\ref{eq:monotonicity}), using dimension reduction arguments one can compute the maximum total dimension of each of the five sets for all the solutions of the family at the same time, see Propositions \ref{prop:total_dim} and \ref{prop:total_dim_2a}. Here, monotonicity is key to get the same results as one would get for a single solution.

Then, for each type of points we use that if $x_0 \in \Gamma(u(\cdot,t_0))$, there exists some $r_0 > 0$ such that $u$ is positive (or identically zero, depending on the case) in one of the following sets
$$\{x \in B_{r_0} : |x-x_0|^\gamma < t - t_0\} \quad \text{or} \quad \{x \in B_{r_0} : |x - x_0|^\gamma < t_0 - t\},$$
and hence there are no other free boundary points there.
This is done via an expansion of the solution at $x_0$ and comparison arguments. The novel results in this step are Propositions \ref{prop:quadratic_normal_cleaning} and \ref{prop:cubic_cleaning}, that deal with quadratic and cubic points, respectively.

Finally, applying Proposition \ref{prop:GMT} one can get an estimate on the size of each of the degenerate strata for almost every solution. For $n \geq 4$, the situation can be summarized as follows, where $\alpha,\gamma \in (0,1)$ are dimensional constants, and $\varepsilon>0$ is an arbitrarily small number.
\begin{table}[ht]
\renewcommand{\arraystretch}{1.235}
\begin{tabular}{|c|c|c|c|}
\hline
Set                     & $\dimH \mathbf{\Gamma}$ & Cleaning exponent & Generic\tablefootnote{In the sense of Remark \ref{obs:prevalence}.} $\dimH\Gamma$  \\ \hline
$\Gamma_2^{\mathrm{o}}$ & $n-1$              & $3-\varepsilon$   & $n-4$             \\ \hline
$\Gamma_2^{\mathrm{a}}$ & $n-2$              & $2-\varepsilon$   & $n-4$             \\  \hline
$\Gamma_3$              & $n-1$              & $2+\gamma$        & $n-3-\gamma$      \\ \hline
$\Gamma_{\geq 7/2}$     & $n-1$              & $5/2-\varepsilon$ & $n - 7/2$         \\ \hline
$\Gamma_*$              & $n-2$              & $1+\alpha$        & $n-3-\alpha$      \\ \hline
\end{tabular}
\renewcommand{\arraystretch}{1}
\end{table}

For $n = 2$ and $n = 3$, the conclusion is that, generically, the free boundary contains no degenerate points.

\section{Dimensional bounds for $\mathbf{\Gamma}_{\geq 2}$ and $\mathbf{\Gamma}_*$}\label{sect:dim_reduction1}
First, we will estimate the size of the sets $\mathbf{\Gamma}_{\geq 2}$ and $\mathbf{\Gamma}_*$ with a dimension reduction argument (recall \eqref{eq:boldgamma}), taking advantage of the fact that the possible global homogeneous solutions of the Signorini problem are completely classified in low dimensions.

In particular, the goal of this section is to prove the following result:

\begin{prop}\label{prop:total_dim}
Let $u : B_1\times[-1,1]\to\R$ be a solution to (\ref{eq:signorini})-(\ref{eq:monotonicity}). Then,
\begin{enumerate}
    \item[(a)] $\dimH(\mathbf{\Gamma}_{\geq 2}) \leq n - 1$ if $n \geq 2$, and $\mathbf{\Gamma}_{\geq 2}$ is discrete if $n = 1$.
    \item[(b)] $\dimH(\mathbf{\Gamma}_*) \leq n - 2$ if $n \geq 3$, $\mathbf{\Gamma}_*$ is discrete if $n = 2$, and it is empty if $n = 1$.
\end{enumerate}
\end{prop}

In order to do it, we first show the following lemma (cf. \cite[Lemma 6.5]{FRS20}).

\begin{lem}\label{lem:dim_red_translation}
Let $u : B_1\times[-1,1]\to\R$ be a solution to (\ref{eq:signorini})-(\ref{eq:monotonicity}), with $0 \in {\Gamma}_{\geq 2}(u(\cdot, 0))$. Let $x_k \in \mathbf{\Gamma}_{\geq 2}$ satisfy $|x_k| \leq r_k$, with $r_k \downarrow 0$,  $t_k := \tau(x_k) \rightarrow 0$, and assume that 
$$\tilde{u}_{r_k} := \frac{u(r_k\cdot,0)}{\|u(r_k\cdot,0)\|_{L^2(\p B_1)}} \rightharpoonup q \ \text{in} \ H^1_{\mathrm{loc}}(\R^{n+1}), \quad y_k := \frac{x_k}{r_k} \rightarrow y_\infty \neq 0, \ \text{and} \ \kappa_k \rightarrow \kappa,$$
where
$$\kappa_k := \phi(0^+,u(x_k+\cdot,t_k)), \quad \kappa := \phi(0^+,u(\cdot,0)),$$
and $q \not\equiv 0$ is a $\kappa$-homogeneous solution to (\ref{eq:signorini}).

Then, $q$ is translation invariant in the direction $y_\infty$.
\end{lem}

\begin{proof}
Let us define $w_k := u(x_k+r_k\cdot,t_k)$ and $w_{k,0}:= u(x_k+r_k\cdot,0)$ so that, for each $k\in \N$, they are ordered in $B_{1/(2r_k)}$ (that is, either $w_k \ge w_{k,0}$ or $w_k \le w_{k,0}$ in $B_{1/(2r_k)}$). Observe that, by assumption, since $\tilde{u}_{r_k} \rightharpoonup q$  and $q \not\equiv  0$,
$$\frac{w_{k,0}}{\|w_{k,0}\|_{L^2(\p B_1)}}  = \frac{\tilde{u}_{r_k}(y_k+\cdot)}{\|\tilde{u}_{r_k}(y_k+\cdot)\|_{L^2(\p B_1)}} \rightharpoonup \frac{q(y_\infty+\cdot)}{\|q(y_\infty+\cdot)\|_{L^2(\p B_1)}}$$
weakly in $H^1\loc$. We now divide the proof into two  steps.

\textit{Step 1.} We first prove that, up to a subsequence,
$$\tilde{w}_k := \frac{w_k}{\|w_k\|_{L^2(\p B_1)}} \rightarrow Q \quad \text{ locally uniformly},$$
for some $Q$ a global $\kappa$-homogeneous solution to the Signorini problem.

Indeed, by the upper semicontinuity and monotonicity of $\phi$, and the fact that $\kappa_k \rightarrow \kappa$, for all $\delta > 0$ there exist $r_\delta > 0$ and $k_\delta \in \N$ such that
$$\phi(r,u(x_k+\cdot,t_k)) \in (\kappa - \delta, \kappa + \delta) \quad \forall r \in (0,r_\delta),\quad \forall k \geq k_\delta,$$
and hence
$$\phi(r,\tilde w_k) \in (\kappa - \delta, \kappa + \delta) \quad \forall r \in (0,r_\delta/r_k), \quad \forall k \geq k_\delta.
$$

In particular, by Lemma \ref{lem:H_bounds},
$$H(R,\tilde{w}_k) \leq R^{2\kappa+2\delta} H(1,\tilde{w}_k) = R^{2\kappa+2\delta} \quad \forall R \in [1,r_\delta/r_k), \quad \forall k \geq k_\delta,$$
maybe with a smaller $r_\delta > 0$ and larger $k_\delta$. Combined with interior Lipschitz estimates \cite[Theorem 1]{AC04}, this implies that $\tilde{w}_k \rightarrow Q$ locally uniformly, up to a subsequence, for some $Q$ a global solution to the thin obstacle problem. Moreover, thanks to the uniform $C^{1,1/2}$ estimates for solutions  \cite{AC04}  we also have that $\phi(r, \tilde w_k)\to \phi(r, Q)$ as $k\to \infty$ for each $r>0$ fixed (observe that $|\p_{n+1}\tilde w_k|^2$ is $C^{1/2}$), and therefore
\[
\phi(r, Q) \in [\kappa -\delta, \kappa+\delta]\quad\forall r > 0.
\]
Since this holds for any $\delta > 0$, Lemma~\ref{lem:ff_increasing} yields that $Q$ is $\kappa$-homogeneous.

\textit{Step 2.} We now show that $y_\infty\cdot\nabla q$ has a constant sign and deduce that $y_\infty\cdot\nabla q = 0$. 

Let $\hat{\varepsilon}_k := \|w_k\|_{L^2(\p B_1)} + \|w_{k,0}\|_{L^2(\p B_1)}$. By the first observation we have
$$\hat{w}_{k,0} := w_{k,0}/\hat{\varepsilon}_k \rightharpoonup bq(y_\infty+\cdot) =: \hat{Q}_0 \text{ weakly in }H^1\loc$$
for some $b \in [0,1]$. Moreover, by Step 1 and up to a subsequence,
$$\hat{w}_k := w_k/\hat{\varepsilon}_k\rightarrow aQ=:\hat Q \quad \text{ locally uniformly},$$
with $a \in [0,1]$. 

We cannot have $a = b = 0$, because it contradicts the fact that  $\|\hat{Q}\|_{L^2(\p B_1)} + \|\hat{Q}_0\|_{L^2(\p B_1)} = 1$. Suppose now that $a = 0$. Then, for each $k\in \N$, $\hat{w}_k$ and $\hat{w}_{k,0}$ are ordered in $B_{1/(2r_k)}$, and therefore $\hat Q_0$ and $\hat Q$ are ordered in $\R^{n+1}$ (that is, either $\hat Q_0\ge \hat Q \equiv 0$ or $\hat Q_0\le \hat Q \equiv 0$ in $\R^{n+1}$). Since $q$ (and then $\hat Q_0$) is a global solution with homogeneity $\kappa \geq 2$, by Lemma \ref{lem:global_signorini} it cannot have constant sign, a contradiction. The same argument with $Q$ gives that $b$ cannot be zero. Hence, $a$ and $b$ are both positive.

If we assume without loss of generality that $\hat{Q}\ge \hat{Q}_0$ and let $z = \lambda x$, by homogeneity we have
$$aQ(x) \geq bq(y_\infty+x) \quad \Rightarrow \quad aQ(z) \geq bq(\lambda y_\infty + z) ~~ \forall\lambda > 0\quad \Rightarrow \quad aQ \geq bq.$$

Since $aQ$ and $bq$ are ordered global solutions of (\ref{eq:signorini}) with homogeneity greater than 1, they are equal by Lemma \ref{lem:FRSA.4}. It follows that
$$bq = aQ \geq bq(y_\infty+\cdot),$$
and by homogeneity again (since $b > 0$)
$$q \geq q(\lambda y_\infty+\cdot) \quad \forall \lambda > 0.$$
Thus, $y_\infty\cdot\nabla q \leq 0$, and applying Lemma \ref{lem:global_signorini}(c), $q$ is invariant in the $y_\infty$ direction.
\end{proof}

We can now give the proof of Proposition \ref{prop:total_dim}.

\begin{proof}[Proof of Proposition \ref{prop:total_dim}]
(a) We will apply Proposition \ref{lem:FRS7.3} to the set $\mathbf{\Gamma}_{\geq 2}$ with the function $f : \mathbf{\Gamma}_{\geq 2} \to \R$ given by
$$f(x_0) = \phi(0^+,u(\cdot,\tau(x_0))).$$

To obtain the desired result, thanks to Lemma \ref{lem:FRS7.3} it suffices to prove the following: for all $x_0 \in \mathbf{\Gamma}_{\geq 2}$ and for all $\varepsilon > 0$, there exists $\rho > 0$ such that for all $r \in (0,\rho)$,
$$B_r(x_0)\cap\mathbf{\Gamma}_{\geq 2}\cap f^{-1}\big([f(x_0)-\rho,f(x_0)+\rho]\big) \subset  \{y : \operatorname{dist}(y,\Pi_{x,r}) \leq \varepsilon r\},$$
where $\Pi_{x,r}$ is a $(n-1)$-dimensional plane passing through $x_0$.

Assume without loss of generality that $x_0 = 0$ and $\tau(x_0) = 0$, and let us prove the statement by contradiction. If such $\rho > 0$ did not exist for some $\varepsilon_0 > 0$, then we would have sequences $r_k \downarrow 0$ and $x_k\pj \in \mathbf{\Gamma}_{\geq 2}\cap B_{r_k}$, $1 \leq j \leq n$, such that
$$y_k\pj := x_k\pj/r_k \rightarrow y_\infty\pj \in \overline{B_1}, \quad \operatorname{dim}(\operatorname{span}(y_\infty\pfirst,\ldots,y_\infty^{(n)})) = n, \quad |f(x_k\pj)-f(0)| \downarrow 0.$$

Let $\tilde{u}_r := u(r\cdot)/H(r,u)^{1/2}$. Then, by \cite[Section 4]{ACS08} $\tilde{u}_r \rightharpoonup q$ along a subsequence, where $q$ is a nonzero $\kappa$-homogeneous global solution to the Signorini problem (\ref{eq:signorini}). Also, since $x_0 \in \mathbf{\Gamma}_{\geq 2}$, $\kappa \geq 2$.

Applying Lemma \ref{lem:dim_red_translation} to the sequences $(x_k\pj, \tau(x_k\pj))$ we deduce that $q$ is translation invariant in the $n$ linearly independent directions $y_\infty\pj$, $1 \leq j \leq n$. It follows that $q$ is a one dimensional nonzero $\kappa$-homogeneous solution to Signorini, with $\kappa \geq 2$, which contradicts the fact that the only possible homogeneities in dimension one are $0$ and~$1$.

(b) Repeating the arguments in (a), but with $1 \leq j \leq n - 1$ instead, we end up with a nonzero $\kappa$-homogeneous two dimensional solution to Signorini, but since $x_0 \in \mathbf{\Gamma}_*$, $\kappa \notin \{1,\frac{3}{2},2,3,\frac{7}{2},4,5,\ldots\}$, contradicting that these are the only possible homogeneities in dimension 2.
\end{proof}

\section{Quadratic points}\label{sect:quadratic}

\subsection{Ordinary quadratic points}
If the next term of the expansion at a quadratic point is at least cubic (that is, we are at an ordinary quadratic point, \eqref{eq:def_ord_anom}), we can adapt the arguments in \cite[Section 9]{FRS20} to improve the cleaning rate up to $3 - \varepsilon$. Hence, we show:

\begin{prop}\label{prop:quadratic_normal_cleaning}
Let $u : B_1 \times [-1,1] \to \R$ be a solution to (\ref{eq:signorini})-(\ref{eq:monotonicity}). Assume that $0 \in \Gamma_2^{\mathrm{o}}(u(\cdot,0))$.

Then, for all $\varepsilon > 0$ there exists $\rho > 0$ such that
$$\{(x,t) \in B_\rho\times[0,1] : t > |x|^{3-\varepsilon}\}\cap\{u = 0\}\cap\{x_{n+1}=0\} = \emptyset.$$
\end{prop}

In order to prove Proposition \ref{prop:quadratic_normal_cleaning}, we first show the following auxiliary lemma.

\begin{lem}\label{lem:FRS9.2}
Let $u : B_1\times[-1,1]\to\R$ be a solution to (\ref{eq:signorini})-(\ref{eq:monotonicity}), with $0 \in \Gamma_2(u(\cdot,0))$. Let $D_r := \p B_r\cap\{|x_{n+1}| > r/2\}$. Then, for every $\varepsilon > 0$,
$$\min\limits_{D_r}h_t := \min\limits_{D_r}\,[u(\cdot,t) - u(\cdot,0)] \geq c_\varepsilon r^\varepsilon t, \quad \forall r \in (0,\rho_\varepsilon), \quad \forall t \in [0,1],$$
for some $c_\varepsilon, \rho_\varepsilon > 0$.
\end{lem}

\begin{proof}
By \cite[Theorem 1.3.6]{GP09}),
$$u(x,0) = p(x) + o(|x|^2),$$
for some nonzero $p \in \Pdos$. Therefore, for all $\delta > 0$ there exists $r_\delta > 0$ such that for all $\rho \in (0,2r_\delta)$,
$$B_1\cap\{u(\rho\cdot,0) = 0\}\cap\{x_{n+1}=0\} \subset C_\delta \hspace{-0.7mm}:= \hspace{-0.7mm}\left\{x \in \R^{n+1} \hspace{-0.4mm}:\hspace{-0.4mm} \operatorname{dist}\left(\hspace{-0.4mm}\frac{x}{|x|},\{p = 0\}\cap\{x_{n+1}=0\}\hspace{-0.4mm}\right)\hspace{-0.7mm} <\hspace{-0.7mm} \delta\right\}\hspace{-0.7mm}.$$

Indeed, let $m$ be the minimum of $p$ in $(\p B_1\cap\{x_{n+1}=0\}) \setminus C_\delta$. Since $p \geq 0$ on the thin space, $m > 0$. Now, choosing $r_\delta$ small enough, for all $\rho < r_\delta$,
$$u(\rho x,0) \geq p(\rho x) - \frac{m}{2}\rho^2|x|^2 = \rho^2|x|^2\left(p\left(\frac{x}{|x|}\right) - \frac{m}{2}\right) > 0,$$
for all $x \in (B_1\cap\{x_{n+1}=0\}) \setminus C_\delta$.

Let now $\varphi_\delta := |x|^{\mu(\delta)}\Phi_\delta(x/|x|)$, where $\Phi_\delta \geq 0$ is the first eigenfunction of the spherical Laplacian on $\p B_1 \setminus C_\delta$, and $\mu(\delta)$ is chosen so that $\varphi_\delta$ is harmonic when positive. Then, $\varphi_\delta$ is a positive harmonic function defined in $\R^n\setminus C_\delta$ vanishing on $\p C_\delta$.

Since $p \not\equiv 0$ and $p$ is a homogeneous quadratic polynomial nonnegative on the thin space, $\{p = 0\}\cap \{x_{n+1} = 0\}$ is a linear space of dimension at most $n - 1$, and in particular has zero harmonic capacity. Therefore, as $\delta \rightarrow 0$, $\mu(\delta) \rightarrow 0$, and we can choose $\delta$ such that $\mu(2\delta) < \varepsilon$. Moreover, choosing $\delta < \frac{1}{4}$, $D_{r_\delta}$ and $C_{2\delta}$ are disjoint.

Notice that $h_t = u(\cdot,t) - u(\cdot,0)$ is harmonic in $\{u(\cdot,0) > 0\}$ and in $B_1\setminus \{x_{n+1}=0\}$. In particular, $h_t$ is harmonic in
$$(B_1\setminus \{x_{n+1}=0\})\cup\left(B_{2r_\delta}\cap\{x_{n+1}=0\}\setminus C_\delta\right).$$
Hence, using the monotonicity assumption (\ref{eq:monotonicity}) and the interior Harnack, there exists $c_\delta > 0$ such that
$$h_t \geq c_\delta t \ \text{on} \ \p B_{r_\delta} \setminus C_{2\delta}.$$

Then, we can use
$$w_t := c_\delta t\frac{\varphi_{2\delta}}{\|\varphi_{2\delta}\|_{L^\infty(\p B_{r_\delta})}}$$
as a lower barrier in $B_{r_\delta}\setminus C_{2\delta}$ because $h_t \geq w_t$ in $\p B_{r_\delta} \setminus C_{2\delta}$ by construction, and $h_t \geq 0$ and $w_t = 0$ on $\p C_{2\delta}$.

Hence,
$$\min\limits_{D_r} h_t \geq \min\limits_{D_r}w_t = cr^{\mu(2\delta)}t \geq cr^\varepsilon t \quad \forall r \in (0,r_\delta),$$
as we wanted to see.
\end{proof}

By means of the previous result, we can now prove the improved cleaning for the ordinary quadratic points.

\begin{proof}[Proof of Proposition \ref{prop:quadratic_normal_cleaning}]
By the definition of $\mathbf{\Gamma}_2^{\mathrm{o}}$, there exists a harmonic quadratic polynomial $p\in \Pdos$ such that
$$|r^{-2}u(r\cdot,0) - p| \leq Cr \ \text{in} \ B_1, \quad \forall r \in (0,1),$$
Let us then bound $v(x) := r^{-2}u(rx,t)$. By Lemma \ref{lem:FRS9.2} and the previous estimates, taking $t \ge r^{3-2\varepsilon}$,
$$v(x) \geq p(x) - Cr + c_\varepsilon r^{\varepsilon - 2}t\chi_{\{|x_{n+1}|>1/2\}} \geq p(x) - Cr + c_\varepsilon r^{1-\varepsilon}\chi_{\{|x_{n+1}|>1/2\}} \ \text{on} \ \p B_1.$$

Let $\varphi$ be a harmonic function in $B_1$ with boundary data $\varphi = \chi_{\{|x_{n+1}| > 1/2\}}$ on $\p B_1$. Then, since $v$ is superharmonic and $p$ is harmonic,
$$v(x) \geq p(x) - Cr + c_\varepsilon r^{1-\varepsilon}\varphi \ \text{in} \ B_1,$$
and using that $\varphi \geq c(n) > 0$ in $B_{1/2}$,
$$v \geq p - Cr + c_\varepsilon c(n) r^{1-\varepsilon} > 0 \ \text{on} \ B_{1/2}\cap\{x_{n+1} = 0\},$$
for sufficiently small $r$, using that $p \geq 0$ on the thin space.
\end{proof}

\subsection{Anomalous quadratic points}
Now we consider the points in the set $\mathbf{\Gamma}_2^{\mathrm{a}}$ (see \eqref{eq:def_ord_anom}). We will use a dimension reduction argument to show that $\dimH (\mathbf{\Gamma}_2^{\mathrm{a}}) \leq n - 2$. Hence, in this subsection we will prove the following proposition.

\begin{prop}\label{prop:total_dim_2a}
Let $u : B_1\times[-1,1]\to\R$ be a solution to (\ref{eq:signorini})-(\ref{eq:monotonicity}). Then, 
$\dimH(\mathbf{\Gamma}_2^{\mathrm{a}}) \leq n - 2$ if $n \geq 3$, $\mathbf{\Gamma}_2^{\mathrm{a}}$ is discrete if $n = 2$, and it is empty if $n = 1$.
\end{prop}

The following lemmas are analogous to the first part of \cite[Section 6]{FRS20} combined with results from \cite{CSV20, FJ21, FR21}. The first one is about the continuity of the first and second blow-ups on the set $\mathbf{\Gamma}_2$.
\begin{lem}\label{lem:FRS6.2}
Let $u : B_1\times[-1,1]\to\R$ be a solution to (\ref{eq:signorini})-(\ref{eq:monotonicity}), and let us denote by $p_{2,x}$ the blow-up of $u(\cdot, \tau(x))$ at $x\in \mathbf{\Gamma}_{\ge 2}$ according to \eqref{eq:p2def}; in particular, $p_{2, x} \equiv 0$ if and only if $x\in \mathbf{\Gamma}_{> 2}$. Then:
\begin{itemize}
    \item[(a)] For all $\rho < 1$, ${\mathbf{\Gamma}}_{\ge 2}\cap\overline{B_\rho}$ is closed. Moreover, given a convergent sequence $\{x_k\} \subset \mathbf{\Gamma}_{\ge 2}\cap\overline{B_\rho}$, $x_k \rightarrow x_\infty$,
    $$p_{2,x_k} \rightarrow p_{2,x_\infty},$$
    where $p_{2,x_\infty} \equiv 0$ if $x_\infty\in \mathbf{\Gamma}_{>2}$.
    \item[(b)] The frequency function
    $$\mathbf{\Gamma}_{\ge 2} \ni x_0 \mapsto \phi\big(0^+,u(x_0+\cdot,\tau(x_0))-p_{2,x_0}\big)$$
    is upper semicontinuous.
\end{itemize}
\end{lem}

\begin{proof}
(a) We first show that if $x_k \in \mathbf{\Gamma}_{\ge 2}$ and $x_k \rightarrow x_\infty$, then $x_\infty\in\mathbf{\Gamma}_{\ge 2}$. Notice that $t_k := \tau(x_k) \rightarrow t_\infty := \tau(x_\infty)$ by Proposition \ref{prop:tau_cont}. Now, by \cite[Proposition 7.1]{CSV20} (or by the frequency gap \cite[Theorem 4]{CSV20} if $x_k\in \mathbf{\Gamma}_{> 2}$) we have
$$\|u(x_k+\cdot,t_k)-p_{2,x_k}\|_{L^\infty(B_r)} \leq r^2\omega(r), \quad \forall r > 0,$$
where $\omega$ is a universal modulus of continuity.

Then, $p_{2,x_k} \rightarrow P$ up to a subsequence for some harmonic 2-homogeneous polynomial $P$ and, by Proposition \ref{prop:tau_cont}, $u(x_k+\cdot,t_k) \rightarrow u(x_\infty+\cdot,t_\infty)$ in $C^0$. Therefore,
$$\|u(x_\infty+\cdot,t_\infty) - P\|_{L^\infty(B_r)} \leq r^2\omega(r), \quad \forall r > 0.$$
It follows that $x_\infty \in \mathbf{\Gamma}_{\ge 2}$ and that $p_{2,x_\infty} = P$. Finally, the estimate can only hold for one unique $P$, and a posteriori we deduce that for any other subsequence, $p_{2,x_{k_j}} \rightarrow P$ up to a subsequence again.

(b) First, we consider the function $\mathbf{\Gamma}_{\ge 2} \ni x_0 \mapsto \phi(r,u(x_0+\cdot,\tau(x_0))-p_{2,x_0})$ for a fixed $r > 0$,
$$\phi(r,u(x_0+\cdot,\tau(x_0))-p_{2,x_0}) = r\frac{\int_{B_r}|\nabla u(x_0+\cdot,\tau(x_0)) - \nabla p_{2,x_0}|^2}{\int_{\p B_r}(u(x_0+\cdot,\tau(x_0))-p_{2,x_0})^2}.$$

Given a convergent sequence $x_k \in \mathbf{\Gamma}_{\ge 2}$, $x_k \rightarrow x_\infty$, using (a) the terms involving the second order polynomial converge. Then, $u(x_k+\cdot,\tau(x_k)) \rightarrow u(x_\infty+\cdot,\tau(x_\infty))$ in $L^\infty$ by the second part of Proposition \ref{prop:tau_cont}. Thus, the quotient is continuous because of the uniform $C^{1,1/2}$ estimates for $u(\cdot,t)$ \cite{AC04} (observe that $|\p_{n+1}u(x_0+\cdot,\tau(x_0))-\p_{n+1}p_{2,x_0}|^2 = |\p_{n+1}u(x_0+\cdot,\tau(x_0))|^2$ is $C^{1/2}$ in $B_r$).

Our desired result now follows by taking the infimum over $r > 0$ of the family of continuous functions $\mathbf{\Gamma}_{\ge 2} \ni x_0 \mapsto \phi(r,u(x_0+\cdot,\tau(x_0))-p_{2,x_0})$ (this is an increaing family in $r > 0$, by Proposition~\ref{prop:ff_increasing_w}).
\end{proof}

Then, we show that points in $\mathbf{\Gamma}_2$ only accumulate in the directions of the null space of the blow-up.

\begin{lem}\label{lem:FRS6.3}
Let $u : B_1\times[-1,1]\to\R$ be a solution to (\ref{eq:signorini})-(\ref{eq:monotonicity}), and let $0 \in \Gamma_2(u(\cdot,0))$. Let $x_k \in \mathbf{\Gamma}_2$ satisfy $|x_k| \downarrow 0$ and $t_k := \tau(x_k) \downarrow 0$. Let $p_{2,k} := p_{2,x_k}$. Then, $p_{2,k} \rightarrow p_2$, with $p_2$ the blow-up of $u(\cdot,0)$ at $0$, and we have
\begin{align*}
    \left\|p_{2,k}-p_2\left(\frac{x_k}{|x_k|}+\cdot\right)\right\|_{L^\infty(B_1)} &\leq C\omega(2|x_k|),\\
    \|p_{2,k}-p_2\|_{L^\infty(B_1)} &\leq C\omega(2|x_k|),
\end{align*}
where $\omega$ is a universal modulus of continuity, and
$$\operatorname{dist}\left(\frac{x_k}{|x_k|},\{p_2 = 0\}\cap\{x_{n+1}=0\}\right) \rightarrow 0 \quad\text{as}\quad k\rightarrow\infty.$$
\end{lem}

\begin{proof}
By Lemma \ref{lem:FRS6.2}~(a), $p_{2,k}\rightarrow p_2$, up to a subsequence. Let $r_k = |x_k|$, so that by \cite[Proposition 7.1]{CSV20} we have
$$\|r_k^{-2}u(x_k+r_kx,t_k)-p_{2,k}(x)\|_{L^\infty(B_2)} \leq 4\omega(2r_k)$$
and
$$\|r_k^{-2}u(r_kx,0)-p_2(x)\|_{L^\infty(B_2)} \leq 4\omega(2r_k).$$

Thus, defining $y_k := x_k/|x_k|$, for all $x \in B_2$ we have the following: if $t_k \leq 0$,
$$-4\omega(2r_k)+p_{2,k}(x)\leq r_k^{-2}u(x_k+r_kx,t_k)\leq r_k^{-2}u(x_k+r_kx,0) \leq 4\omega(2r_k)+p_2(y_k+x),$$
and if $t_k \geq 0$,
$$4\omega(2r_k)+p_{2,k}(x)\geq r_k^{-2}u(x_k+r_kx,t_k)\geq r_k^{-2}u(x_k+r_kx,0) \geq -4\omega(2r_k)+p_2(y_k+x).$$

Assume without loss of generality that $t_k \geq 0$ and consider the function $q(x) = p_{2,k}(x) - p_2(y_k+x) + 8\omega(2r_k)$. On the one hand, $q$ is nonnegative and harmonic in $B_2$. On the other hand, since $p_2(y_k+\cdot) \geq 0$ on $\{x_{n+1} = 0\}$, $q(0) \leq 8\omega(2r_k)$. Then, by the Harnack inequality, $0 \leq q \leq C\omega(2r_k)$ in $B_1$.

Consequently,
$$\|p_{2,k}-p_2(y_k+\cdot)\|_{L^2(\p B_1)} \leq C\|p_{2,k}-p_2(y_k+\cdot)\|_{L^\infty(B_1)} \leq C\omega(2r_k).$$

Finally, $p_{2,k}-p_2$ is 2-homogeneous and harmonic, and $p_2 - p_2(y_k+\cdot)$ is affine. Therefore, they are orthogonal. Hence, when $k \rightarrow \infty$,
$$\|p_{2,k}-p_2\|^2_{L^2(\p B_1)} + \|p_2-p_2(y_k+\cdot)\|^2_{L^2(\p B_1)} = \|p_{2,k}-p_2(y_k+\cdot)\|^2_{L^2(\p B_1)} \rightarrow 0.$$
In particular, $\|p_2 - p_2(y_k+\cdot)\|_{L^2(\p B_1)} \rightarrow 0$, and it follows that $\operatorname{dist}(y_k, \{p_2 = 0\}\cap\{x_{n+1}=0\}) \rightarrow 0$.
\end{proof}

The following auxiliary lemma plays a similar role to Lemma \ref{lem:dim_red_translation}, but for the second blow-up at anomalous quadratic points.

\begin{lem}\label{lem:FRS6.4}
Let $u : B_1\times[-1,1]\to\R$ be a solution to (\ref{eq:signorini})-(\ref{eq:monotonicity}), let $0 \in \Gamma_2^{\mathrm{a}}(u(\cdot,0))$. Let $x_k \in \mathbf{\Gamma}_2^{\mathrm{a}}$ satisfy $|x_k| \leq r_k$ with $r_k \downarrow 0$ and $t_k := \tau(x_k) \rightarrow 0$. Assume that
$$\tilde{w}_{r_k} := \frac{w(r_k\cdot)}{\|w(r_k\cdot)\|_{L^2(\p B_1)}} \rightharpoonup q \ \text{in} \ H^1_{\mathrm{loc}}(\R^{n+1}) \ \text{for} \ w := u(\cdot,0) - p_2, \quad y_k := \frac{x_k}{r_k} \rightarrow y_\infty,$$
where $p_2$ is the blow-up of $u(\cdot,0)$ at $0$ and $y_\infty \neq 0$.

Then, $q(y_\infty) = 0$.
\end{lem}

\begin{proof}
Let us define $v_k := u(x_k+r_k\cdot,t_k) - p_2(r_k\cdot) = v_k\pfirst + v_k\psecond + v_k\pthird$, where
\begin{align*}
    v_k\pfirst &:= u(x_k+r_k\cdot,t_k) - u(x_k+r_k\cdot,0),\\
    v_k\psecond &:= u(x_k+r_k\cdot,0) - p_2(x_k+r_k\cdot),\\
    v_k\pthird &:= p_2(x_k+r_k\cdot) - p_2(r_k\cdot).
\end{align*}

Observe that $\tilde{w}_{r_k} \rightharpoonup q$, and $\|q(y_k+\cdot)\|_{L^2(\p B_1)} \neq 0$ because $q$ is homogeneous and nonzero by Proposition \ref{prop:quadratic_stratification}. Therefore,
$$\frac{v_k\psecond}{\|v_k\psecond\|_{L^2(\p B_1)}} = \frac{w_{r_k}(y_k+\cdot)}{\|w_{r_k}(y_k+\cdot)\|_{L^2(\p B_1)}} = \frac{\tilde{w}_{r_k}(y_k+\cdot)}{\|\tilde{w}_{r_k}(y_k+\cdot)\|_{L^2(\p B_1)}} \rightharpoonup \frac{q(y_\infty+\cdot)}{\|q(y_\infty+\cdot)\|_{L^2(\p B_1)}},$$
weakly in $H^1\loc$.

On the other hand, notice that the zero level set of a nonnegative homogeneous quadratic polynomial coincides with the linear space of invariant directions. Let $L := \{p_2 = 0\}\cap\{x_{n+1} = 0\}$. Then, $L$ is a linear subspace of dimension at most $n - 1$ because $p_2 \not\equiv 0$ on the thin space. Now, $p_2(y_\infty) = 0$ by the second part of Lemma \ref{lem:FRS6.3}, and denoting $z_k$ the orthogonal projections of $y_k$ onto $L$,
$$\frac{v_k\pthird}{\|v_k\pthird\|_{L^2(\p B_1)}} = \frac{p_2(y_k+\cdot)-p_2}{\|p_2(y_k+\cdot)-p_2\|_{L^2(\p B_1)}} = \frac{p_2(y_k-z_k+\cdot)-p_2}{\|p_2(y_k-z_k+\cdot)-p_2\|_{L^2(\p B_1)}} \rightharpoonup \nabla p_2\cdot e,$$
weakly in $H^1\loc$, up to a subsequence, because $y_k - z_k \rightarrow 0$, and for some  non-zero $e \in L^\perp$.

We now divide the proof into three steps.

\textit{Step 1.} We prove that
$$\tilde{v}_k := \frac{v_k}{\|v_k\|_{L^2(\p B_1)}} \rightharpoonup Q \quad \text{in} \ H^1_{\mathrm{loc}}(\R^{n+1})$$
for some $Q$ with polynomial growth.

By Proposition \ref{prop:tau_cont} and the monotonicity of $\phi$, there exist $r_0 > 0$ and $k_0 \in \N$ such that, for $M := \phi(0^+, u(\cdot,0) - p_2) + 1$, we have

$$\phi(r,u(x_k+\cdot,t_k) - p_2) \leq M \quad \forall r \in (0,r_0), \quad \forall k \geq k_0$$
and equivalently
$$\phi(r,\tilde{v}_k) = \phi(r,v_k) \leq M \quad \forall r \in (0,r_0/r_k), \quad \forall k \geq k_0.$$
Applying Lemma \ref{lem:H_bounds_2} to $v_k$, we obtain
$$H(R,\tilde{v}_k) \leq CR^{2M+1}H(1,\tilde{v}_k) = CR^{2M+1} \quad \forall R \in [1,r_0/r_k), \quad \forall k \geq k_0,$$
maybe with a smaller $r_0 > 0$, and then $\|\tilde{v}_k\|_{H^1(B_R)} \leq C(R)$.

By compactness, it follows that $\tilde{v}_k \rightharpoonup Q$ in $H^1_{\mathrm{loc}}(\R^{n+1})$, up to a subsequence.

\textit{Step 2.} Observe that $q$ is harmonic by Proposition \ref{prop:quadratic_stratification}. We now prove that $Q$ is harmonic as well and grows at most quadratically at the origin.

First, $\Delta\tilde{v}_k \leq 0$ in $B_{1/r_k}$. Moreover, by \cite[Proposition 7.1]{CSV20},
$$\|u(x_k+\rho\cdot,t_k) - p_{2,x_k}(\rho\cdot)\|_{L^1(\p B_1)} \leq \rho^2\omega(\rho),$$
with $\omega(\rho) \rightarrow 0$ as $\rho \rightarrow 0$, and hence
$$\|u(x_k+\rho\cdot,t_k) - p_{2,x_k}(\rho\cdot)\|_{L^\infty(B_1)} \leq C\rho^2\omega(\rho).$$
Furthermore, for $R \geq 1$, substituting $\rho = Rr_k \leq 1$, 
$$\|u(x_k+r_k\cdot,t_k) - p_{2,x_k}(r_k\cdot)\|_{L^\infty(B_R)} \leq C(Rr_k)^2\omega(Rr_k),$$
and for any $x \in B_R\cap\{u(x_k+r_kx,t_k)=0\}$, using that the polynomial is $2$-homogeneous,
$$p_{2, x_k}(x) \leq CR^2\omega(Rr_k) \Rightarrow p_2(x) \leq CR^2\omega(Rr_k),$$
by Lemma \ref{lem:FRS6.3}.

Then, since $p_2$ grows quadratically away from its zero set,
\begin{align*}
    &B_R\cap\{u(x_k+r_k\cdot,t_k) = 0\}\cap\{x_{n+1}=0\} \subset\\
    &\qquad\big\{y \in B_R : \operatorname{dist}(y,L) \leq CR\left[\omega(Rr_k)\right]^{1/2}\big\}\cap\{x_{n+1}=0\},
\end{align*}
and the right hand side tends to $0$ as $k \rightarrow \infty$ for any fixed $R$. This shows that
$$\sup\{\operatorname{dist}(x,L) : x \in B_R\cap\{u(x_k+r_k\cdot,t_k) = 0\}\}\cap\{x_{n+1}=0\} \downarrow 0,$$
and it follows that the weak limit of the sequence of nonpositive measures $\Delta\tilde{v}_k$ will be supported on $L$.

Finally, since $L$ is a linear space of at most dimension $n - 1$, given any test function $\xi \in C^\infty_c(\R^{n+1})$, it can be approximated in $H^1$ norm by $\xi_j\rightarrow\xi$ that vanish on $L$. Hence,
$$\int\nabla Q \cdot \nabla \xi = \lim\limits_{j \to \infty}\int\nabla Q\cdot\nabla \xi_j = -\lim\limits_{j\to\infty}\int\xi_j\Delta Q = 0,$$
and it follows that $Q$ is harmonic. Observe, also, that by Lemma \ref{lem:H_bounds_2}, given that $x_k \in \mathbf{\Gamma}_2$,
$$H(\rho,v_k) \leq \rho^4H(1,v_k) \quad \forall \rho \in (0,1),$$
and hence in the limit $\|Q(\rho\cdot)\|_{L^2(\p B_1)}^2 = H(\rho,Q) \leq \rho^4$ for all $\rho \in (0,1)$, so $Q$ is at most quadratic at the origin.

\textit{Step 3.} We finally prove that $q(y_\infty) = 0$.

First, let $\hat{\varepsilon}_k := \|v_k\pfirst\|_{L^2(\p B_1)} + \|v_k\psecond\|_{L^2(\p B_1)} + \|v_k\pthird\|_{L^2(\p B_1)}$ and $\hat{v}_k := v_k/\hat{\varepsilon}_k$. By Step 1 we have $\hat{v}_k \rightharpoonup \hat{Q} = aQ$ for some $a \in [0,1]$. Moreover, by the first observations,
$$v_k\psecond/\hat{\varepsilon}_k \rightharpoonup bq(y_\infty+\cdot) := \hat{Q}\psecond, \quad v_k\pthird/\hat{\varepsilon}_k \rightharpoonup c\nabla p_2\cdot e := \hat{Q}\pthird,$$
weakly in $H^1\loc$, for some $b,c \geq 0$.

Then, the following limit is well defined:
$$\hat{Q}\pfirst := \lim\limits_k v_k\pfirst/\hat{\varepsilon}_k = \lim\limits_k v_k/\hat{\varepsilon}_k - \lim\limits_k v_k\psecond/\hat{\varepsilon}_k - \lim\limits_k v_k\pthird/\hat{\varepsilon}_k,$$
and it has a constant sign because all the $v_k\pfirst$ do. Since $\hat{Q}$, $\hat{Q}\psecond$ and $\hat{Q}\pthird$ are harmonic, $\hat{Q}\pfirst$ must be harmonic as well, and by the Liouville theorem, it must be constant. Hence,
$$\hat{Q} = C + bq(y_\infty+\cdot) + c\nabla p_2\cdot e,$$
and, by the definition of $\hat{\varepsilon}_k$,
$$C\|1\|_{L^2(\p B_1)} + b\|q(y_\infty+\cdot)\|_{L^2(\p B_1)} + c\|\nabla p_2\cdot e\|_{L^2(\p B_1)} = 1.$$

If $\hat{Q} \equiv 0$, since $q$ is quadratic, we would have $b = 0$. Then, since $\nabla p_2\cdot e$ is linear, it would follow that all the terms in the sum are zero, a contradiction.

Therefore, $\hat{Q} \not\equiv 0$, i.e. $a \neq 0$. Since $Q$ grows at most quadratically, $b > 0$ and $\nabla Q(0) = 0$. Hence,
\begin{align*}
    0 &= y_\infty\cdot\nabla\hat{Q}(0) = by_\infty\cdot\nabla q(y_\infty) + cy_\infty\cdot\nabla(\nabla p_2\cdot e)(0)
    = 2bq(y_\infty) + 0,
\end{align*}
where we used that $q$ is 2-homogeneous and $y_\infty \in \{p_2 = 0\}$, and it follows that $q(y_\infty) = 0$, as required.
\end{proof}

Now we are ready to prove our dimensional bound on $\mathbf{\Gamma}_2^{\mathrm{a}}$. 

\begin{proof}[Proof of Proposition \ref{prop:total_dim_2a}]
We need to prove that, for any $\beta > n - 2$, the set $\mathbf{\Gamma}_2^{\mathrm{a}}$ has zero $\beta$-dimensional Hausdorff measure. Assume by contradiction that
$$\mathcal{H}^\beta(\mathbf{\Gamma}_2^{\mathrm{a}}) > 0.$$

Then, by the basic properties of Hausdorff measures (see \cite[2.10.19(2)]{Fed69}) there exists a point $x_0 \in \mathbf{\Gamma}_2^{\mathrm{a}}$ (let us assume $x_0 = 0$ without loss of generality), a sequence $r_k \downarrow 0$ and a set $A \subset \overline{B_1}$, with $\mathcal{H}^\beta(A) > 0$, such that for every point $y \in A$, there is a sequence $x_k \in \mathbf{\Gamma}_2^{\mathrm{a}}$ such that $x_k/r_k \rightarrow y$.

Let $w = u(\cdot,0)-p_2$, $w_r = w(r\cdot)$ and $\tilde{w}_r = w_r/H(1,w_r)^{1/2}$. Then, by assumption,
$$\tilde{w}_{r_k} \rightharpoonup q \ \text{in} \ H^1\loc$$
up to a subsequence, where $q$ is a 2-homogeneous harmonic polynomial.

Furthermore, by Lemma \ref{lem:FRS6.4} we have $A \subset \{q = 0\}\cap\{p_2 = 0\}\cap\{x_{n+1}=0\}$. Then, since $\mathcal{H}^\beta(A) > 0$, with $\beta > n - 2$, the only possibility is that $\operatorname{dim}(\{p_2 = 0\}\cap\{x_{n+1}=0\}) = n - 1$, and that $q \equiv 0$ on $\{p_2 = 0\}\cap\{x_{n+1}=0\}$. Hence, after a change of variables, we may assume $p_2(x',0) = x_1^2$, and therefore $p_2(x) = x_1^2 - x_{n+1}^2$, and $q(x) = x_1(a\cdot x) - a_1x_{n+1}^2$.

Now, by the first part of Lemma \ref{lem:orthogonality},
$$0 = \int_{\p B_1}(x_1^2-x_{n+1}^2)(x_1(a\cdot x)-a_1x_{n+1}^2) = a_1\int_{\p B_1}(x_1^2-x_{n+1}^2)^2,$$
where we used that, for $i > 1$, $x_1x_i$ is odd with respect to $x_1$ and $x_1^2-x_{n+1}^2$ is even. It follows that $a_1 = 0$.

On the other hand, using the second part of Lemma \ref{lem:orthogonality}, and letting $p = C(x_1^2+x_i^2-2x_{n+1}^2) + a_ix_1x_i$ with $i > 1$, and $C > 0$ large enough such that $p(x',0) \geq 0$,
$$0 \geq \int_{\p B_1}(C(x_1^2+x_i^2-2x_{n+1}^2)+a_ix_1x_i)(x_1(a\cdot x)) = a_i^2\int_{\p B_1}x_1^2x_i^2,$$
using again the odd and even symmetries of the terms involved. We conclude that $a_i = 0$ for all $i = 2,\ldots,n$. But then it follows that $q \equiv 0$, a contradiction.
\end{proof}

\section{Cubic points}\label{sect:cubic}
In this section, we improve the cleaning rate of the cubic points using a barrier argument combining \cite[Lemma 9.4]{FRS20} with Theorem \ref{thm:cubic_approx} and the Hopf-type estimate in Lemma~\ref{lem:hopf}.

\begin{prop}\label{prop:cubic_cleaning}
Let $u : B_1\times[-1,1] \to \R$ be a solution to (\ref{eq:signorini})-(\ref{eq:monotonicity}), with $0 \in \Gamma_3(u(\cdot,0))$. Then, there exist some $r_0,c_0 > 0$ such that, for all $t \in (-1,0]$,
$$\{x \in B_{r_0} : |x|^{2+\gamma} < -c_0t\}\cap\Gamma(u(\cdot,t)) = \emptyset,$$
for some $\gamma > 0$ only depending on $n$.
\end{prop}

\begin{proof}
Let $c_0, \gamma > 0$ to be chosen later. We will prove that there exists $0 < r_0 < \frac{1}{8}$ such that for all $r \in (0,r_0)$, and $t$ with $-c_0t \geq r^{2+\gamma}$,
$$u(\cdot,t) \equiv 0 \ \text{on} \ B_{r}\cap\{x_{n+1} = 0\},$$
and in particular there are no free boundary points there.

By Theorem \ref{thm:cubic_approx} and Lemma \ref{lem:cubic_blowup},
$$\|r^{-3}u(r\cdot,0) - p_3\|_{L^\infty(B_2)} \leq Cr^\alpha, \quad p_3(x',x_{n+1}) = |x_{n+1}|(ax_{n+1}^2 - x'\cdot Ax'),$$
with $a \geq 0$ and $A$ symmetric and nonnegative definite.

Let us then bound $v(x) := r^{-3}u(rx,t)$. By Lemma \ref{lem:hopf} (after reversing $t$) and the previous estimates,
\begin{align*}
    v(x) &\leq r^{-3}u(rx,0) - cr^{-3}|t||rx_{n+1}|\leq a|x_{n+1}|^3 + Cr^\alpha - C_1r^\gamma|x_{n+1}| \ \text{in} \ B_2,
\end{align*}
where $C_1 = c/c_0$. Now, given $z' \in \R^n$ with $|z'| < 1$, and $\delta \geq 0$, we introduce the barrier
$$\psi_{z',\delta}(x',x_{n+1}) = -(n+1)x_{n+1}^2+(x'-z')^2+\delta.$$

Let $z = (z',0)$, and let $s = (Cr^\alpha)^{1/2}$, which is smaller than $1$ for sufficiently small $r$. We will prove that $v \leq \psi_{z',\delta}$ in $B_s(z)$. First, given $x \in \p B_s(z)$, using that $(x'-z')^2 = s^2-x_{n+1}^2$, it suffices to see that
$$a|x_{n+1}|^3+Cr^\alpha-C_1r^\gamma|x_{n+1}| \leq -(n+2)x_{n+1}^2+s^2 \ \text{for} \ |x_{n+1}| \leq s,$$
which after choosing $s = (Cr^\alpha)^{1/2}$ becomes
$$C_1r^\gamma|x_{n+1}| \geq a|x_{n+1}|^3+(n+2)x_{n+1}^2 \ \text{for} \ |x_{n+1}| \leq (Cr^\alpha)^{1/2},$$
that is satisfied choosing $\gamma = \alpha/2$ and a sufficiently large $C_1$ (i.e., a sufficiently small $c_0$).

Let us assume that there exists $\delta > 0$ such that $\psi_{z',\delta}$ touches $v$ from above in $\overline{B}_s(z)$ at $x_0$. Observe that $x_0 \in B_s(z)$ because $\psi_{z',\delta} > v$ on $\p B_s(z)$ for all positive $\delta$. Now, if $x_0 \notin \{x_{n+1}= 0, v = 0\}$, $\Delta v (x_0) = 0$ and $\Delta \psi_{z',\delta} = -2$, we have a superharmonic function touching a harmonic function from above, which is a contradiction. On the other hand, if $x_0$ belongs to the contact set,
$$0 = v(x_0) = \psi_{z',\delta}(x_0) = (x_0'-z')^2+\delta > 0,$$
a contradiction as well. Therefore, the only possibility is that $v \leq \psi_{z',\delta}$ in $B_s(z)$ for all $\delta > 0$, and in particular $v(z) \leq 0$.

Repeating the argument for all $z \in B_1\cap\{x_{n+1}=0\}$, we obtain that $v \equiv 0$ on $B_1\cap\{x_{n+1}=0\}$, which is the same as $u(\cdot,t) \equiv 0$ on $B_{r}\cap\{x_{n+1} = 0\}$.
\end{proof}

\section{Proof of Theorem \ref{thm:main}}\label{sect:conclusion}
We  take advantage of the following stratification of the degenerate set to compute our estimates:
$$\mathbf{\operatorname{Deg}} = \mathbf{\Gamma}_2^{\mathrm{o}}\cup\mathbf{\Gamma}_2^{\mathrm{a}}\cup\mathbf{\Gamma}_3\cup\mathbf{\Gamma}_{\geq 7/2}\cup\mathbf{\Gamma}_*.$$

We can now apply Proposition \ref{prop:GMT} to obtain generic dimensional estimates for all of these sets.

\renewcommand{\labelitemii}{$\bullet$}

\begin{prop}\label{prop:generic}
Let $u : B_1\times[-1,1]\to\R$ be a solution to (\ref{eq:signorini})-(\ref{eq:monotonicity}). Let $\pi_2 : (x,t) \mapsto t$ be the standard projection. Then, there exist $\alpha, \gamma > 0$, depending only on $n$, such that:
\begin{itemize}
    \item[(a)] If $n = 1$,
        \begin{itemize}
        \item $\mathbf{\Gamma}_2^{\mathrm{o}}$ is discrete,
        \item $\mathbf{\Gamma}_2^{\mathrm{a}} = \emptyset$,
        \item $\mathbf{\Gamma}_3 = \emptyset$,
        \item $\mathbf{\Gamma}_{\geq 7/2}$ is discrete,
        \item $\mathbf{\Gamma}_* = \emptyset$.
    \end{itemize}
    \item[(b)] If $n = 2$,
        \begin{itemize}
        \item $\dimH (\pi_2(\mathbf{\Gamma}_2^{\mathrm{o}})) \leq 1/3$,
        \item $\mathbf{\Gamma}_2^{\mathrm{a}}$ is discrete,
        \item $\dimH (\pi_2(\mathbf{\Gamma}_3)) \leq 1/(2+\gamma)$,
        \item $\dimH (\pi_2(\mathbf{\Gamma}_{\geq 7/2})) \leq 2/5$,
        \item $\mathbf{\Gamma}_*$ is discrete.
    \end{itemize}
    \item[(c)] If $n = 3$, 
    \begin{itemize}
        \item $\dimH (\pi_2(\mathbf{\Gamma}_2^{\mathrm{o}})) \leq 2/3$,
        \item $\dimH (\pi_2(\mathbf{\Gamma}_2^{\mathrm{a}})) \leq 1/2$,
        \item $\dimH (\pi_2(\mathbf{\Gamma}_3)) \leq 2/(2+\gamma)$,
        \item $\dimH (\pi_2(\mathbf{\Gamma}_{\geq 7/2})) \leq 4/5$,
        \item $\dimH (\pi_2(\mathbf{\Gamma}_*))\leq 1/(1+\alpha)$.
    \end{itemize}
    \item[(d)] If $n \geq 4$, for $\mathcal{H}^1$-a.e. $t \in [-1,1]$,
    \begin{itemize}
        \item $\dimH (\Gamma_2^{\mathrm{o}}(u(\cdot,t))) \leq n - 4$,
        \item $\dimH (\Gamma_2^{\mathrm{a}}(u(\cdot,t))) \leq n - 4$,
        \item $\dimH (\Gamma_3(u(\cdot,t))) \leq n - 3 - \gamma$,
        \item $\dimH (\Gamma_{\geq7/2}(u(\cdot,t))) \leq n - \frac{7}{2}$,
        \item $\dimH (\Gamma_*(u(\cdot,t))) \leq n - 3 - \alpha$.
    \end{itemize}
\end{itemize}
\end{prop}

\begin{proof}
For each of the sets considered, we combine a total dimension estimate with a \textit{cleaning} result.
\begin{itemize}
    \item For $\mathbf{\Gamma}_2^{\mathrm{o}}$, by Proposition \ref{prop:total_dim}(a), $\dimH(\mathbf{\Gamma}_2^{\mathrm{o}}) \leq n - 1$, and $\mathbf{\Gamma}_2^{\mathrm{o}}$ is discrete when $n = 1$. By Proposition \ref{prop:quadratic_normal_cleaning}, for all $x_0 \in \mathbf{\Gamma}_2^{\mathrm{o}}$ and for all $\varepsilon > 0$, there exist $r_0, c > 0$ such that
$$\{x \in B_{r_0} : |x - x_0|^{3 - \varepsilon} < (t-\tau(x_0))\}\cap\mathbf{\Gamma}_2^{\mathrm{o}} = \emptyset.$$

    \item For $\mathbf{\Gamma}_2^{\mathrm{a}}$, by Proposition \ref{prop:total_dim_2a}, $\dimH(\mathbf{\Gamma}_2^{\mathrm{a}}) \leq n - 2$, $\mathbf{\Gamma}_2^{\mathrm{a}}$ is discrete when $n = 2$, and it is empty when $n = 1$. By Proposition \ref{prop:cleaning_old}, for all $x_0 \in \mathbf{\Gamma}_2$ and for all $\varepsilon > 0$, there exist $r_0, c > 0$ such that
$$\{x \in B_{r_0} : |x - x_0|^{2 - \varepsilon} < (t-\tau(x_0))\}\cap\mathbf{\Gamma}_2 = \emptyset.$$

    \item For $\mathbf{\Gamma}_3$, by Proposition \ref{prop:total_dim}(a), $\dimH(\mathbf{\Gamma}_3) \leq n - 1$, and $\mathbf{\Gamma}_3$ is discrete when $n = 1$. By Proposition \ref{prop:cubic_cleaning}, for all $x_0 \in \mathbf{\Gamma}_3$, there exist $r_0, c > 0$ such that
$$\{x \in B_{r_0} : |x - x_0|^{2+\gamma} < -c(t-\tau(x_0))\}\cap\mathbf{\Gamma}_3 = \emptyset,$$
and after changing $t$ by $-t$, for all $\varepsilon > 0$ there exists $r_1 > 0$ such that for all $r \in (0,r_1)$,
$$B_r(x_0)\cap\{(x,t) : x \in \Gamma_3(u(\cdot,t))\} = \emptyset$$
for all $t > \tau(x_0) + c^{-1}r^{2+\gamma} \geq \tau(x_0) + r^{2+\gamma-\varepsilon}.$

    \item For the set $\mathbf{\Gamma}_{\geq 7/2}$, by Proposition \ref{prop:total_dim}(a), $\dimH(\mathbf{\Gamma}_{\geq 7/2}) \leq n - 1$, and $\mathbf{\Gamma}_{\geq 7/2}$ is discrete when $n = 1$. By Proposition \ref{prop:cleaning_old}, for all $x_0 \in \mathbf{\Gamma}_{\geq 7/2}$ and for all $\varepsilon > 0$, there exists $r_0 > 0$ such that
$$\{x \in B_{r_0} : |x - x_0|^{5/2 - \varepsilon} < (t-\tau(x_0))\}\cap\mathbf{\Gamma}_{\geq 7/2} = \emptyset.$$

    \item Finally, for $\mathbf{\Gamma}_*$, by Proposition \ref{prop:total_dim}(b), $\dimH(\mathbf{\Gamma}_*) \leq n - 2$, $\mathbf{\Gamma}_*$ is discrete when $n = 2$, and it is empty when $n = 1$. Then, thanks to \cite[Theorem 4]{CSV20}, the order of the points in $\mathbf{\Gamma}_*$ is $\kappa \geq 2+\alpha$ for some dimensional $\alpha > 0$. Applying Proposition \ref{prop:cleaning_old} as in the previous case, for all $x_0 \in \mathbf{\Gamma}_*$ and for all $\varepsilon > 0$, there exists $r_0 > 0$ such that
$$\{x \in B_{r_0} : |x - x_0|^{1 + \alpha - \varepsilon} < (t-\tau(x_0))\}\cap\mathbf{\Gamma}_* = \emptyset.$$
\end{itemize}

The conclusions follow now by Proposition \ref{prop:GMT}.
\end{proof}

Finally, we can prove our main results.

\begin{proof}[Proof of Theorem \ref{thm:main}]
It is a direct consequence of Proposition \ref{prop:generic}.
\end{proof}

\begin{proof}[Proof of Conjecture \ref{conj:schaeffer}]
It is a direct consequence of Proposition \ref{prop:generic}. The smoothness of the free boundary follows from \cite{KPS15, DS16}.
\end{proof}

\end{document}